\documentclass[a4paper,10pt]{amsart}
\usepackage{amsthm,amsmath,amsfonts,amssymb}
\usepackage{bm}
\usepackage{mathtools,mathrsfs}
\usepackage{todonotes}
\usepackage{tikz}
\usepackage[pdfdisplaydoctitle,colorlinks,breaklinks,urlcolor=blue,linkcolor=blue,citecolor=blue]{hyperref} 
\usepackage[nameinlink,capitalise]{cleveref}

\hypersetup{
  pdftitle   = {Fluctuations of Point Vortex Ensembles at Small Negative Temperature},
  pdfauthor  = {Francesco Grotto},
  pdfsubject = {Gaussian fluctuations of canonical Gibbs point vortex ensembles},
  pdfkeywords = {
    point vortices,
    negative temperature,
    cluster expansion,
    2D Euler equations,
    equilibrium fluctuations
  }
}

\AddToHook{env/theorem/begin}
  {\crefalias{section}{theorem}}

\AddToHook{env/lemma/begin}
  {\crefalias{theorem}{lemma}}

\AddToHook{env/proposition/begin}
  {\crefalias{theorem}{proposition}}

\AddToHook{env/corollary/begin}
  {\crefalias{theorem}{corollary}}

\AddToHook{env/definition/begin}
  {\crefalias{theorem}{definition}}

\newcommand{\D}{\mathcal{D}}

\newcommand{\E}{\mathcal{E}}

\newcommand{\e}{\mathrm e}
\newcommand{\F}{\mathcal{F}}

\newcommand{\K}{\mathcal{K}}

\newcommand{\U}{\mathcal{U}}
\newcommand{\T}{\mathbb{T}}

\DeclareMathOperator{\trace}{Tr}

\DeclareMathOperator{\ipi}{1PI}

\newcommand{\dettwo}{\det\nolimits_2}

\renewcommand{\epsilon}{\varepsilon}

\newcommand{\one}{\bm{1}}

\newcommand{\set}[1]{\left\{#1\right\}}
\newcommand{\pa}[1]{\left(#1\right)}

\newcommand{\abs}[1]{\left|#1\right|}
\newcommand{\norm}[1]{\left\|#1\right\|}
\newcommand{\brak}[1]{\left\langle#1\right\rangle}

\newcommand{\wick}[1]{:\mathrel{#1}:}

\newcommand{\ssum}{\mathop{\textstyle\sum}\nolimits}

\tikzset{
  smallgraph/.style={
    scale=0.35,
    baseline=(current bounding box.center),
    every node/.style={circle, fill, inner sep=1.15pt},
    every path/.style={line width=0.45pt}
  }
}

\newcommand{\ingraph}[1]{%
  \mathord{\vcenter{\hbox{#1}}}%
}

\newcommand{\Vof}[1]{A_N\!\left(#1\right)}

\newcommand{\gvee}{%
  \ingraph{%
    \begin{tikzpicture}[smallgraph]
      \node (a) at (0,0) {};
      \node (b) at (1.2,0) {};
      \node (c) at (0.6,0.9) {};
      \draw (a)--(c)--(b);
    \end{tikzpicture}%
  }%
}

\newcommand{\gtriangle}{%
  \ingraph{%
    \begin{tikzpicture}[smallgraph]
      \node (a) at (0,0) {};
      \node (b) at (1.2,0) {};
      \node (c) at (0.6,0.9) {};
      \draw (a)--(b)--(c)--(a);
    \end{tikzpicture}%
  }%
}

\newcommand{\gsquare}{%
  \ingraph{%
    \begin{tikzpicture}[smallgraph]
      \node (a) at (0,0) {};
      \node (b) at (1,0) {};
      \node (c) at (1,1) {};
      \node (d) at (0,1) {};
      \draw (a)--(b)--(c)--(d)--(a);
    \end{tikzpicture}%
  }%
}

\newcommand{\gbarsquare}{%
  \ingraph{%
    \begin{tikzpicture}[smallgraph]
      \node (a) at (0,0) {};
      \node (b) at (1,0) {};
      \node (c) at (1,1) {};
      \node (d) at (0,1) {};
      \draw (a)--(b)--(c)--(d)--(a);
      \draw (a)--(c);
    \end{tikzpicture}%
  }%
}

\newcommand{\gcompletesquare}{%
  \ingraph{%
    \begin{tikzpicture}[smallgraph]
      \node (a) at (0,0) {};
      \node (b) at (1,0) {};
      \node (c) at (1,1) {};
      \node (d) at (0,1) {};
      \draw (a)--(b)--(c)--(d)--(a);
      \draw (a)--(c);
      \draw (b)--(d);
    \end{tikzpicture}%
  }%
}

\newcommand{\gncycle}{%
  \ingraph{%
    \begin{tikzpicture}[smallgraph]
      \def\r{0.7}
      \def\rdots{1}

      \node[label=left:{\tiny $1$}]  (v1) at (150:\r) {};
      \node[label=above:{\tiny $2$}] (v2) at (90:\r) {};
      \node[label=right:{\tiny $3$}] (v3) at (30:\r) {};
      \node[label=left:{\tiny $k$}]  (vn) at (210:\r) {};

      \draw (v1) arc[start angle=150,end angle=90,radius=\r];
      \draw (v2) arc[start angle=90,end angle=30,radius=\r];
      \draw (v3) arc[start angle=30,end angle=-150,radius=\r];
      \draw (vn) arc[start angle=210,end angle=150,radius=\r];

      \node[inner sep=0.42pt] at (-28:\rdots) {};
      \node[inner sep=0.42pt] at (-45:\rdots) {};
      \node[inner sep=0.42pt] at (-62:\rdots) {};
    \end{tikzpicture}%
  }%
}

\newtheorem{theorem}{Theorem}

\newtheorem{lemma}[theorem]{Lemma}
\newtheorem{proposition}[theorem]{Proposition}

\theoremstyle{remark}

\numberwithin{equation}{section}

\newenvironment{acknowledgements}{%
  
  \begin{abstract}
}{%
  \end{abstract}
}

\title[Negative Temperature Point Vortex]{Fluctuations of Point Vortex Ensembles\\ at Small Negative Inverse Temperature}
\author[F. Grotto]{Francesco Grotto}
\address{Università di Pisa, Dipartimento di Matematica, 5 Largo Bruno Pontecorvo, 56127 Pisa, Italia.}
\email{francesco.grotto at unipi.it}

\date\today

\begin{document}

\begin{abstract}
    The vorticity distribution associated to canonical Gibbs point vortex ensembles under mean field scaling has Gaussian fluctuations for small negative inverse temperature. The perturbative argument is based on a cluster expansion of the partition function.
\end{abstract}

\maketitle

\section{Negative Temperature Point Vortex Ensembles}\label{sec:introduction}

Let $\T\simeq [0,1]^2$ be the flat 2D torus and $G$ the Laplacian Green function $-\Delta G=\delta_0-1$ under the zero average condition $\int_{\T}G(x)\,d x=0$. The point vortex dynamics on $\T$ is a system of $N$ singular ODEs,
\begin{equation*}
    \dot x_i=\sum_{j\neq i} (-\partial_2,\partial_1)G(x_i-x_j),\qquad i=1,\dots, N,
\end{equation*}
such that their (centered) empirical measure $\omega=\sum^N \delta_{x_j}-Ndx$ is a measure-valued weak solution of 2D Euler equations for an inviscid incompressible fluid with vorticity $\omega=\nabla\times u$. It is a Hamiltonian system in the coordinates of points $(x_{i,1},x_{i,2})$, with Hamiltonian function
\begin{equation*}
    H_N(x_1,\dots,x_N)=\sum_{1\leq i<j\leq N}G(x_i-x_j),
\end{equation*}
The statistical mechanics of point vortices is a classical heuristic proposed by Onsager \cite{onsager1949} for explaining the formation of coherent structures in 2D fluids. Onsager observed that for microcanonical distribution $\delta(H_N=E)$, finiteness of phase space volume forces entropy to decrease for large enough energy $E$.
As a consequence, clusters of vortices having intensities of the same sign become statistically predominant for large $E$. In the canonical ensemble, the same phenomenon is reproduced by negative values of the inverse temperature. The physical relevance of this fact is not limited to fluid dynamics \cite{purcell1951,abraham2017}.
As already observed by Onsager the question is then how to relate the statistical mechanics of point vortices with 2D Euler equations. Mean field scaling limits are a standard answer and the object of extensive literature. 

In this note I shall focus on the (mean field rescaled) canonical Gibbs ensembles of a single species of vortices on $\T$ at negative inverse temperature,
\begin{equation*}
    d\nu_{N,\beta}= \frac1{Z_N(\beta)}e^{\frac\beta{N}H_N(x_1,\dots,x_N)}dx^N,\quad
    Z_N(\beta)= \int_{\T^N}e^{\frac\beta{N}H_N}dx^N,
\end{equation*}
which is well-defined ($Z_N(\beta)<\infty$) for $\beta< 8\pi$. Notice that the usual inverse temperature is $-\beta$ in this notation, which I employ since I will only consider negative inverse temperature. It is known \cite{cagliotiI,cagliotiII,lionsbook} that for $\beta<8\pi$,
\begin{equation}\label{eq:mfasymp}
    \lim_{N\to \infty} \frac1N \log Z_N(\beta)=\sup
      \left\{
      \frac\beta2\int_{\T}\rho\,(G\ast\rho)\,dx
      -\int_{\T}\rho\log\rho\,dx \left| \rho\geq0,\, \int_{\T}\rho\,dx=1
      \right.\right\},
\end{equation}
where the maximizer satisfies the mean field equation
\begin{equation*}
    \rho(x)
    =\frac{\exp\pa{\beta(G\ast\rho)(x)}}
    {\int_{\T}\exp\pa{\beta(G\ast\rho)(y)}\,dy}.
\end{equation*}
For more general geometric domains and vortex circulations $\omega=\sum_i\gamma_i\delta_{x_i}$, solutions of the associated mean field equation can behave in various ways, but in the specific one-species case on the torus I am considering the unique solution is the flat profile $\rho=1$ for all $\beta<8\pi$ \cite{gumora,gugugu}, that is $\log Z_N(\beta)=o(N)$ as $N\to \infty$.
This means that the canonical and microcanonical ensembles are only equivalent at null energy $E=0$ (\emph{cf.} \cite{eyinkspohn,kiesslinglebowitz}), and apparently the canonical ensemble does not witness vortex clustering. I will argue that this is not the case: Gaussian fluctuations around the (trivial) average distribution correspond to energy-enstrophy equilibrium ensembles of the 2D Euler equations exhibiting concentration for negative inverse temperature.
The latter fact is well understood for positive inverse temperatures, the aim of this note is to prove it for small negative inverse temperatures by means of a perturbative argument.

The idea is best understood focusing on the distribution of $H_N$, which can be regarded as a U-statistic of vortex positions. The normalized empirical measure of vortices
\begin{equation*}
    \omega_N=\frac1{\sqrt{N}}\pa{\ssum^N_{j=1} \delta_{x_j}-Ndx}
\end{equation*}
converges in law on $\D'(\T)$ to the zero-averaged space white noise $\xi$ on $\T$, that is $\xi$ is the centered Gaussian process on $L^2_0(\T)$ with $E[\xi(\phi)\xi(\psi)]=\int \phi\psi dx$. A classical limit theorem for U-statistics then implies that $H_N/N$ converges to the random variable $\frac12:\brak{\xi,G\ast\xi}:$ in the second Wiener chaos of $\xi$ \cite{dynkin}.
In the terms of fluid dynamics, recall that the 2D Euler equations in vorticity form read
\begin{equation*}
    \partial_t \omega+u\cdot \nabla\omega=0,\quad \nabla\cdot u=0,
\end{equation*}
and there are two quadratic first integrals: energy and enstrophy,
\begin{equation*}
    E=\frac12 \int_\T |u(x)|^2 dx=\frac12 \int_\T \omega(x)(G\ast \omega)(x)dx,\qquad \E=\frac12 \int_\T \omega(x)^2dx.
\end{equation*}
White noise $\xi$ is (a multiple of) the \emph{enstrophy ensemble} formally defined by $\frac{1}{Z}e^{-\E(\omega)}d\omega$, but one can consider more generally the Gaussian measure
\begin{equation*}
    d\mu_{\beta}(\omega)=\frac1Z e^{\beta E(\omega)-\E(\omega)}d\omega
\end{equation*}
usually called the \emph{energy-enstrophy ensemble}, as an invariant measure of 2D Euler equations \cite{albeverioribeiro,albeveriocruzeiro}. The measure $\mu_\beta$ is absolutely continuous with respect to white noise, because $:E(\omega):$ is a well-defined random variable under the Gaussian measure $\mu_0$ as a Wick product. In fact, the law of $:E(\omega):$ under $\mu_0$ is exactly the limit in law of $H_N/N$ (see \cite{grottoromito}). Moreover, $:E:$ is exponentially integrable,
\begin{equation*}
    Z_\beta=\int e^{\beta :E:(\omega)}d\mu_0(\omega)=\dettwo(I-\beta G)^{-1/2}<\infty,\quad \beta<4\pi^2.
\end{equation*}
and $Z_\beta^{-1}e^{\beta :E:}=\frac{d\mu_\beta}{d\mu_0}$ in that range (given by the spectral gap of the Laplacian on $\T$). Here and below I shall consistently identify kernels and their associated integral operators on $L_0^2(\T)$, powers and iterated compositions are distinguished writing $G^n$ and $G^{\circ n}$.
Notice that $G$ is not trace-class, and $\dettwo$ is the Carleman-Fredholm determinant, related to the Fredholm determinant by
\begin{equation}\label{eq:fredholm}
    \dettwo(I-\beta G)=\det\bigl((I-\beta G)\e^{\beta G}\bigr)=
    \exp\pa{-\ssum_{k\ge2}\tfrac{\beta^k}{k}\trace(G^{\circ k})}.
\end{equation}

As a consequence, the law of $\omega_N$ under the canonical ensemble distribution $\nu_{N,\beta'}$ for positions $x_j$ converges in law on $\D'(\T)$ as $N\to \infty$ to samples of $\mu_{\beta'}$ provided that the following uniform integrability condition is satisfied:
\begin{equation}\label{eq:ui}
    \sup_{N\geq 2} Z_N(\beta)<\infty.
\end{equation}
for some $\beta>\beta'$. For $\beta\leq 0$ (positive inverse temperature) \eqref{eq:ui} was established in \cite{grottoromito,grottoromitosd,grottoluongoromito}, and it is a considerable improvement over the asymptotic \eqref{eq:mfasymp} obtained by variational methods.
This note is devoted to the proof of:

\begin{theorem}\label{thm:main}
There exists $0<\beta_0<8\pi$ such that, for every $0<\beta<\beta_0$,
\begin{equation*}
    \lim_{N\to\infty} Z_N(\beta)=\dettwo(I-\beta G)^{-1/2}.
\end{equation*}
As a consequence, the law of $\omega_N$ on $\D'(\T)$ under the positions' distribution $\nu_{N,\beta}(dx^N)$ weakly converges to $\mu_\beta$ for $\beta<\beta_0$.
\end{theorem}

The second statement is a direct corollary of Vitali's theorem and the stated asymptotic, which is proved in \cref{sec:cluster} by means of a cluster expansion of the partition function $Z_N(\beta)$.
\cref{sec:conclusions} collects remarks on extensions and generalizations, the possibility of going beyond the perturbative approach to this result, and a comparison with the statistical mechanics of 2D Coulomb gas.

\section{Cluster expansion}\label{sec:cluster}

As customary, in the following $C$ denotes a positive constant possibly different in every occurrence, depending only on any subscripts. Define the \emph{Mayer function}
\begin{equation*}
    K_N(x)=I_N^{-1}\e^{\frac\beta{N}G(x)}-1,\qquad 
    I_N=\int_{\T}\e^{\frac\beta{N}G(x)}\,d x,
\end{equation*}
and observe that $K_N$ is zero-averaged, $\int_{\T}K_N(x)\,d x=0$, and
\begin{equation*}
    I_N=1+\frac{\beta^2}{2N^2}\int_{\T}G(x)^2\,d x+O(N^{-3}),
\end{equation*}
by Taylor expansion. Write the partition function in terms of Mayer functions,
\begin{equation*}
    Z_N(\beta)
    =I_N^{N(N-1)/2}W_N(\beta),\quad 
    W_N(\beta)
    =\int_{\T^N} \prod_{1\leq i<j\leq N}\bigl(1+K_N(x_i-x_j)\bigr) dx^N,
\end{equation*}
where
\begin{equation}\label{eq:I}
    \binom N2\log I_N=\frac{\beta^2}{4}\trace(G^{\circ 2})+O(N^{-1}).
\end{equation}
gives to leading order the $k=2$ term in \eqref{eq:fredholm}, because $\int_{\T}G(x)^2\,d x=\trace(G^{\circ 2})$.

Expanding the product in the definition of $W_N(\beta)$ leads to a sum of many terms, for which diagram notation is a convenient and natural bookkeeping device. For fixed $N$ I represent the integrands involving $k$ Mayer functions by a (possibly disconnected) graph $\Gamma$ with $N=v(\Gamma)$ vertices and $k=e(\Gamma)$ edges, call them \emph{cluster diagrams} and their integral the \emph{amplitude} $A_N(\Gamma)$ of the diagram $\Gamma$. The function $K_N$ acts as the propagator. The diagrams for $k=0,1$ are trivial, for $k=2$ there is only
\begin{equation*}
    \Vof{\gvee}
    =\int_{\T^3} K_N(x_1-x_2)K_N(x_2-x_3)dx_1dx_2dx_3=0.
\end{equation*}
Already from this simple example it is clear that diagrams with legs (vertices of degree 1) have null amplitude. More generally, any diagram with a connected component that can be disconnected by removing a single edge will have null amplitude, because one can first integrate over variables of one of the resulting connected components and reduce the starting amplitude to that of a diagram with a leg, by translation invariance.
The diagrams with four or less vertices whose amplitude is not null are:
\begin{equation*}
    \gtriangle,\qquad
    \gsquare, \quad \gbarsquare, \quad \gcompletesquare,
\end{equation*}
and the first disconnected diagram with nontrivial amplitude consists of two copies of the 3-cycle.

Let $\ipi(n)$ be the class of connected diagrams with $n$ vertices that cannot be disconnected by removing a single edge, that is \emph{1-particle irreducible} diagrams (hence the acronym) in the language of particle physics. In terms of graph theory, these are the connected simple graphs with vertex set $\set{1,\dots,n}$ that do not have bridges.
By the previous consideration,
\begin{equation}\label{eq:WNexpansion}
    W_N
    =
    \sum_{\ell\ge0}\frac1{\ell!}
    \sum_{\substack{n_1,\dots,n_\ell\ge2\\
    n_1+\cdots+n_\ell\leq N}}
    \frac{(N)_{n_1+\cdots+n_\ell}}{n_1!\cdots n_\ell!}
    \prod_{j=1}^{\ell}\sum_{\Gamma\in \ipi(n_j)}A_N(\Gamma),
\end{equation}
where the term $\ell=0$ is equal to one and $(a)_b$ is the decreasing factorial. The first order term in the asymptotic expansion of $W_N$ as $N\to \infty$ is entirely determined by cycle diagrams,
\begin{equation*}
    \Vof{\gncycle}=\trace(K_N^{\circ k}),
\end{equation*}
Indeed, there are $\frac12 (k-1)!$ possible cycles on $k$ vertices, and given $\beta_*<8\pi$ there exist constants $C_{\beta_*},N_{\beta_*}$ such that for $N\geq N_{\beta_*}$,
\begin{equation*}
    \norm{K_N-\frac\beta{N}G}_{\operatorname{Hilbert-Schmidt}}\leq \frac{C_{\beta_*}}{N^2},
\end{equation*}
(this follows from an elementary expansion as in the forthcoming \cref{lem:centered-edge-xspace}) therefore for each fixed $k\geq 3$ the total contribution of cycles to $W_N$ is
\begin{equation*}
    \binom Nk\frac{(k-1)!}{2}\trace(K_N^{\circ k})
    =
    \frac{\beta^k}{2k}\trace(G^{\circ k}) (1+o(1)),\quad N\to\infty,
\end{equation*}
exactly matching the statement of \cref{thm:main} (the prefactor gives the term $k=2$).

The contribution of all the other diagrams will be controlled by indexing them with unicyclic skeletons, that is I will estimate the amplitude of a diagram by considering a spanning subgraph containing a single cycle. Retaining a cycle is crucial, as considering instead spanning trees would not produce tight enough estimates. Let me introduce right away the symbol $\U(n)$ for the set of connected graphs on $n$ vertices that have a single cycle.

Before moving to the proof of \cref{thm:main}, let me recall a basic exponential estimate on partition functions. It can be directly derived from the arguments in any of the previous contributions \cite{cagliotiI,lionsbook,kiessling}.

\begin{lemma}\label{lem:mean-field-log-integrability}
For every $\beta_*<8\pi$ there exists $C_{\beta_*}<\infty$ such that, for all $0\leq \beta\leq \beta_*$ and $2\leq n\leq N$,
\begin{equation}\label{eq:mean-field-log-integrability}
    \int_{\T^n}\exp\pa{\frac{\beta}{N}\sum_{1\leq i<j\leq n}|G(x_i-x_j)|} dx^n
    \leq C_{\beta_*}^{n}.
\end{equation}
\end{lemma}

I shall also make repeated use of some elementary estimates on propagators.

\begin{lemma}\label{lem:centered-edge-xspace}
For every $0\leq \beta_*<8\pi$ there exists $C_{\beta_*}<\infty$ such that, for all $0\leq \beta \leq \beta_*$, $N\ge2$, and $x\in\T$,
\begin{align}\notag
    |\log I_N|&\leq C_{\beta_*}\frac{\beta^2}{N^2}\\
    1+|K_N(x)|
    &\leq 
    \exp\left\{
        \frac{\beta}{N}|G(x)|
        +C_{\beta_*}\frac{\beta^2}{N^2}
    \right\},
    \label{eq:1+K_Nestimate}
    \\
    |K_N(x)|
    &\leq 
    \frac{C_{\beta_*}\beta}{N}
    \bigl(1+|G(x)|\bigr)
    \exp\left\{
        \frac{\beta}{N}|G(x)|
        +C_{\beta_*}\frac{\beta^2}{N^2}
    \right\}.
    \notag
\end{align}
\end{lemma}

\begin{proof}
The first inequality follows from Taylor expansion in $\beta/N$, because the second derivative is uniformly integrable thanks to $\beta/N\leq \beta_*/2<4\pi$. For the other ones, write
\begin{equation*}
    K_N(x)=\e^{\frac\beta N G(x)-\log I_N}-1,
\end{equation*}
and apply $1+|\e^u-1|\leq \e^{|u|}$ and $|\e^u-1|\leq |u|\e^{|u|}$.
\end{proof}

The forthcoming Lemmas collect the relevant estimates on diagram amplitudes. 

\begin{lemma}\label{lem:unicyclic-log-moment}
For every $1\leq p<\infty$ there exists $C_p<\infty$ such that, for every $U\in\U(n)$,
\begin{equation}\label{eq:unicyclic-log-moment}
    \int_{\T^n}
    \prod_{(i,j)\in E(U)}
    \bigl(1+|G(x_i-x_j)|\bigr)^p
    dx^n
    \leq C_p^n.
\end{equation}
\end{lemma}

\begin{proof}
Let $H_p=(1+|G|)^p\in L^2(\T)$.  Integrating out all legs of
$U$ contributes one factor $\|H_p\|_{L^1}$ per removed edge and reduces the
graph to its unique cycle, of some length $r\ge3$.  The remaining cycle integral is $\trace(H_p^{\circ r})$. As a convolution operator $H_p$ is Hilbert--Schmidt, so $H_p^{\circ 2}$ is trace class and
\begin{align*}
    \bigl|\trace(H_p^{\circ r})\bigr|
    \leq \trace(H_p^{\circ 2})
    \|H_p\|_{L^2\to L^2}^{r-2}
    \leq 
    \|H_p\|_{L^2}^2\|H_p\|_{L^1}^{r-2}.
\end{align*}
The complete integral is therefore bounded by
\begin{equation*}
\|H_p\|_{L^1}^{n-r}\|H_p\|_{L^2}^2\|H_p\|_{L^1}^{r-2}
    =
    \|H_p\|_{L^2}^2\|H_p\|_{L^1}^{n-2}
    \leq C_p^n.\qedhere
\end{equation*}
\end{proof}

\begin{lemma}\label{lem:uniform-combinatorics}
Fix $\beta_*<8\pi$.  There exists $C_{\beta_*}<\infty$ such that, for
all $\beta\leq \beta_*$, all $N\ge2$, and all $2\leq n\leq N$,
\begin{equation}\label{eq:connected-activity-xspace}
    B_{N,n}
    :=
    \binom Nn
    \sum_{\Gamma\in\ipi(n)}|A_N(\Gamma)|
    \leq 
    Cn^{-1/2}\rho^n,
    \qquad
    \rho:=C_{\beta_*}\beta.
\end{equation}
As a consequence, denoting by $\K_N$ the complete graph with $N$ vertices, if $\rho<1$,
\begin{equation}\label{eq:absolute-WN-bound}
    \sup_{N\ge2} \sum_{\Gamma\subseteq \K_N}|A_N(\Gamma)|
    \leq \exp\pa{C\sum_{n\ge2}n^{-1/2}\rho^n}<\infty.
\end{equation}
\end{lemma}

\begin{proof}
The case $n=2$ is empty.  
Assign to each $\Gamma\in \ipi(n)$ a spanning subgraph $U(\Gamma)\in \U(n)$ (it exists because $\Gamma$ has at least $n$ edges, one can take a spanning tree and add an edge).
Grouping diagrams $\Gamma$ according to their chosen skeleton, and denoting by $x_e$ the variable difference associated to an edge $e$,
\begin{multline*}
    \sum_{\Gamma \in \ipi(n)}\left|A_N(\Gamma)\right| 
    \leq \sum_{U \in \U(n)} \sum_{\substack{\Gamma \in \ipi(n) \\ U(\Gamma)=U}} \int_{\T^n} \prod_{e \in E(\Gamma)}\left|K_N\left(x_e\right)\right| d x^n\\
    \leq \sum_{U \in \U(n)} \int_{\T^n} \prod_{e \in E(U)}\left|K_N\left(x_e\right)\right| \prod_{e \notin E(U)}\left(1+\left|K_N\left(x_e\right)\right|\right) d x^n\\
    \leq \left(\frac{C_{\beta_*}\beta}{N}\right)^n \sum_{U \in \U(n)}
    \int_{\T^n}
    \prod_{e\in E(U)}(1+|G(x_e)|)
    \exp\pa{\frac{\beta}{N}\sum_{i<j}|G(x_i-x_j)|}dx^n,
\end{multline*}
the last inequality from \cref{lem:centered-edge-xspace}.
Now choose $q>1$ close to one so that $q\beta_*<8\pi$, put $p=q/(q-1)$ and apply H\"older inequality:
\begin{multline*}
    \int_{\T^n}
    \prod_{e\in E(U)}\bigl(1+|G(x_e)|\bigr)
    \exp\pa{\frac{\beta}{N}\sum_{i<j}|G(x_i-x_j)|}
    dx^n\\
    \leq \pa{\int_{\T^n}\prod_{e\in E(U)}\bigl(1+|G(x_e)|\bigr)^p dx^n}^{1/p}
    \pa{\int_{\T^n}\exp\pa{\frac{q\beta}{N}\sum_{i<j}|G(x_i-x_j)|}
    dx^n}^{1/q},
\end{multline*}
the right-hand side now is bounded by $C_{\beta_*}^n$ by \cref{lem:unicyclic-log-moment,lem:mean-field-log-integrability}. As for the sum over $U$, observe that a unicyclic graph can be generated by choosing a labeled spanning tree and
then one additional edge, so by Cayley's formula $|\U(n)|\leq n^{n-2}\binom n2 \leq \frac12 n^n$.
Together with Stirling's formula, this gives 
\begin{equation*}
B_{N,n}
    \leq 
    \frac{n^n}{2n!}
    \bigl(C_{\beta_*}\beta\bigr)^n
    \leq 
    Cn^{-1/2}\rho^n,
\end{equation*}
and the combination of the estimates obtained so far gives the first statement.

As for the second statement, observe first that if $\ipi(S)$ are the 1PI diagrams on the set of vertices $S$,
\begin{equation*}
    \sum_{\Gamma\subseteq \K_N}|A_N(\Gamma)|
    =\sum_{\F} \prod_{S\in\mathcal F}\sum_{\gamma\in\ipi(S)}|A_N(\gamma)|,
\end{equation*}
where $\sum_\F$ denotes the sum over families $\F$ of disjoint vertex subsets of $\K_N$.
Dropping the disjointness constraint and then allowing repeated sets yields
\begin{multline*}
    \sum_{\Gamma\subseteq \K_N}|A_N(\Gamma)|
    \leq 
    \sum_{\ell\ge0}\frac1{\ell!}
    \left(
        \sum_{S\subseteq \K_n}\sum_{\gamma\in\ipi(S)}|A_N(\gamma)|
    \right)^\ell\\
    =\exp\pa{\sum_{n=2}^N \binom Nn \sum_{\gamma\in\ipi(n)}|A_N(\gamma)|}
    \leq \exp\pa{C\sum_{n\ge2}n^{-1/2}\rho^n}.\qedhere
\end{multline*}
\end{proof}

\begin{lemma}\label{lem:fixed-graph-scaling}
For every finite simple graph $\Gamma$ with $n$ vertices and every $\beta_*<8\pi$,
there exists $C_{\Gamma,\beta_*}<\infty$ such that, for all sufficiently
large $N$ and $\beta\leq \beta_*$,
\begin{equation*}
    |A_N(\Gamma)| \leq C_{\Gamma,\beta_*}
    \left(\frac{\beta}{N}\right)^{e(\Gamma)}.
\end{equation*}
\end{lemma}

\begin{proof}
Applying \cref{lem:centered-edge-xspace} to every edge gives
\begin{gather*}
    |A_N(\Gamma)|
    \leq 
    \left(\frac{C_{\beta_*}\beta}{N}e^{C_{\beta_*}\frac{\beta^2}{N^2}}\right)^{e(\Gamma)}
    J_{N,\beta},\\
    J_{N,\beta}
    :=
    \int_{\T^n}
    \prod_{e\in E(\Gamma)}
        \bigl(1+|G(x_e)|\bigr)
    \exp\left\{
        \frac{\beta}{N}
        \sum_{e\in E(\Gamma)}|G(x_e)|
    \right\}
    dx^n.
\end{gather*}
For every $\varepsilon>0$, the elementary inequality
$1+t\leq C_\epsilon e^{\epsilon t}$, $t\ge0$, yields
\begin{equation*}
    J_{N,\beta}\leq C_{\Gamma,\epsilon}\int_{(\T)^n}\exp\pa{\left(\epsilon+\frac{\beta}{N}\right) \sum_{a\in E(\Gamma)}|G(x_a)|}dx^n.
\end{equation*}
Choose $\lambda>0$ so that $\lambda n<8\pi$, and then choose
$0<\varepsilon<\lambda/2$.  Since $\beta\leq \beta_*$, there is
$N_0=N_0(\Gamma,\beta_*)$ such that, for every $N\ge N_0$, $\varepsilon+\frac{\beta}{N}\leq \lambda$. Moreover, $E(\Gamma)\subseteq E(\K_n)$. Hence
\begin{equation*}
    J_{N,\beta} \leq C_{\Gamma,\varepsilon} \int_{\T^n} \exp\pa{\lambda\sum_{1\leq i<j\leq n}|G(x_i-x_j)|} dx^n.
\end{equation*}
Apply \cref{lem:mean-field-log-integrability} with particle number $n=N$ and inverse temperature $\widetilde\beta=\lambda n<8\pi$.  Since $\widetilde\beta/n=\lambda$, the last integral is finite and depends only on $\Gamma$ and $\lambda$, therefore $J_{N,\beta}$ is uniformly bounded in $N\ge N_0$, $\beta\leq \beta_*$ by a finite constant. The factor $\exp(C_{\beta_*}e(\Gamma)\beta^2/N^2)$ may be absorbed into the final constant.
\end{proof}

Recall that 
\begin{equation*}
    W_N=
    \sum_{\ell\ge0}\frac1{\ell!}
    \sum_{\substack{n_1,\dots,n_\ell\ge2\\
    n_1+\cdots+n_\ell\leq N}}
    \frac{(N)_{n_1+\cdots+n_\ell}}{n_1!\cdots n_\ell!}
    \prod_{j=1}^{\ell}\theta_{N,n_j},
    \qquad 
    \theta_{N,n}:=\sum_{\Gamma\in\ipi(n)}A_N(\Gamma).
\end{equation*}
The above estimates now allow to pass to the limit this expression.

\begin{proof}[Proof of \cref{thm:main}]
    For $\beta_*<8\pi$ let $C_{\beta_*}$ as in \cref{lem:uniform-combinatorics} and take $\beta_0=\min(\beta_*,(2C_{\beta_*})^{-1})$ so that for $\beta<\beta_0$, $\rho=C_{\beta_*}\beta<2$.
    
    For fixed $n$, every $\ipi(n)$ diagram has at least $n$ edges.  If $e(\Gamma)\ge n+1$, by \cref{lem:fixed-graph-scaling},
    \begin{equation*}
    N^n|A_N(\Gamma)|
    \leq 
    C_{\Gamma,\beta_*}\beta^{e(\Gamma)}
    N^{n-e(\Gamma)}=o(1),\qquad N\to \infty.
    \end{equation*}
    If $e(\Gamma)=n$, $\Gamma$ is a cycle. There are $(n-1)!/2$ unoriented cycles with $n$ vertices, and each has amplitude $\trace(K_N^{\circ n})$.  Since $\ipi(2)=\varnothing$, for every fixed $n\ge2$,
    \begin{equation}\label{eq:theta-pointwise-limit}
    c_n=\lim_{N\to \infty }N^n\theta_{N,n}=
    \begin{cases}
        0, & n=2,\\[2mm]
        \displaystyle
        \tfrac{1}{2}(n-1)!\beta^n\trace(G^{\circ n}), & n\ge3.
    \end{cases}
    \end{equation}
    It remains to justify passage of the limit through \eqref{eq:WNexpansion}.  For a fixed $\ell$ and fixed $n_1,\dots,n_\ell$, let $v:=n_1+\cdots+n_\ell$. Using $(N)_v\leq \prod_{j=1}^{\ell}(N)_{n_j}$, the triangle inequality, and \cref{lem:uniform-combinatorics} letting $b_n=Cn^{-1/2}\rho^n$,
    \begin{equation*}
    \abs{\frac1{\ell!}
    \one_{\{v\leq N\}}
    \frac{(N)_v}{n_1!\cdots n_\ell!}
    \prod_{j=1}^{\ell}\theta_{N,n_j}}
    \leq 
    \frac1{\ell!}
    \prod_{j=1}^{\ell}
    \pa{\binom N{n_j}
        \sum_{\Gamma\in\ipi(n_j)}|A_N(\Gamma)|}
    \leq 
    \frac1{\ell!}\prod_{j=1}^{\ell}b_{n_j},
    \end{equation*}
    where right-hand side is summable,
    \begin{equation*}
    \sum_{\ell\ge0}\frac1{\ell!}
    \sum_{n_1,\dots,n_\ell\ge2}
    \prod_{j=1}^{\ell}b_{n_j}
    =\exp\pa{\sum_{n\ge2}b_n}<\infty.
    \end{equation*}
    On the other hand, \eqref{eq:theta-pointwise-limit} implies, for each fixed $\ell$-uple,
    \begin{equation*}
    \lim_{N\to \infty }\one_{\{v\leq N\}}
    \frac{(N)_v}{n_1!\cdots n_\ell!}
    \prod_{j=1}^{\ell}\theta_{N,n_j}
    =
    \prod_{j=1}^{\ell}\frac{c_{n_j}}{n_j!}.
    \end{equation*}
    Dominated convergence in \eqref{eq:WNexpansion} therefore yields
    \begin{align*}
    \lim_{N\to \infty} W_N(\beta)
    =
    \sum_{\ell\ge0}\frac1{\ell!}
    \sum_{n_1,\dots,n_\ell\ge2}
    \prod_{j=1}^{\ell}\frac{c_{n_j}}{n_j!}
    =
    \exp\pa{\sum_{n\ge2}\frac{c_n}{n!}}
    =\exp\pa{\sum_{n\ge3}
        \frac{\beta^n}{2n}\trace(G^{\circ n})}.
    \end{align*}
    Recalling \eqref{eq:I}, the prefactor in $Z_N=I_N^{N(N-1)/2}W_N$ completes the series at the exponent and therefore the proof.
\end{proof}

\section{Further Considerations}\label{sec:conclusions}

Let me first observe that \cref{thm:main} is easily adapted to neutral two-species point vortex ensembles. Let $N=2m$ and denote by $y_1,\dots,y_m$ and $z_1,\dots,z_m$ the positions of vortices with circulations $+1$ and $-1$, respectively. Set
\begin{gather*}
    H_m(y)=\sum_{1\leq i<j\leq m}G(y_i-y_j),\qquad
    C_m(y,z)=\sum_{i,j=1}^mG(y_i-z_j),\\
    H_N^\pm(y,z)=H_m(y)+H_m(z)-C_m(y,z),
\end{gather*}
and define
\begin{equation*}
    Z_N^\pm(\beta)
    =\int_{\T^N}\exp\pa{\frac\beta N H_N^\pm(y,z)}dy^m dz^m,
    \qquad
    \omega_N^\pm
    =\frac1{\sqrt N}\pa{\sum_{i=1}^m\delta_{y_i}-\sum_{j=1}^m\delta_{z_j}}.
\end{equation*}
The associated mean field variational problem is
\begin{multline*}
    \sup
    \left\{
    \frac\beta8\int_\T(\rho_+-\rho_-)
        G\ast(\rho_+-\rho_-)\,dx
    -\frac12\int_\T\pa{\rho_+\log\rho_++\rho_-\log\rho_-}\,dx
    \right.\\
    \left| \rho_\pm\geq0,\, \int_\T\rho_\pm dx=1 \right\},
\end{multline*}
and the Euler-Lagrange equations form a coupled pair of mean field equations,
\begin{equation*}
    \rho_\pm=\frac{e^{\pm \beta\phi/2}}{\int e^{\pm \beta\phi/2} dx},
    \qquad -\Delta\phi=\rho_+-\rho_-.
\end{equation*}
The flat pair $(\rho_+,\rho_-)=(1,1)$ is always a critical point, but it is the unique maximizer only up to a certain threshold $0\leq \beta\leq \beta_\star$ for which only an estimate $\beta_\star\leq 2\pi^2$ is available, \cite[\S 5.3]{lionsbook}. In any case, \cref{thm:main} implies an analogue perturbative result and the needed uniform control of the two-species partition function follows directly from the one-species estimates: by H\"older's inequality for conjugate exponents $p,q>1$,
\begin{multline}\label{eq:two-species-partition-bound}
    Z_N^\pm(\beta)
    \leq
    \left[\int_{\T^N}
    \exp\pa{\frac{2p\beta}{N}\bigl(H_m(y)+H_m(z)\bigr)}dy^m dz^m
    \right]^{1/p}\\
    \times
    \left[\int_{\T^N}
    \exp\pa{-\frac{q\beta}{N}(H_m(y)+H_m(z)+C_m(y,z))}dy^m dz^m
    \right]^{1/q}\\
    =Z_m(p\beta)^{2/p}Z_N(-q\beta)^{1/q}.
\end{multline}
For every $\beta<\beta_0$ one can choose $p>1$ so close to one that $p\beta<\beta_0$, then \cref{thm:main} bounds the first factor, while the uniform positive-temperature estimate of \cite{grottoromito} bounds the second one. Hence
\begin{equation}\label{eq:two-species-ui}
    \sup_{N=2m\geq 2}Z_N^\pm(\beta)<\infty,
    \qquad 0<\beta<\beta_0.
\end{equation}

\begin{proposition}\label{prop:two-species}
For every $0<\beta<\beta_0$,
\begin{equation*}
    \lim_{N\to\infty}Z_N^\pm(\beta)
    =\dettwo(I-\beta G)^{-1/2}.
\end{equation*}
The law of $\omega_N^\pm$ under the canonical ensemble distribution of positions with partition function $Z_N^\pm(\beta)$, $0<\beta<\min(\beta_0,\beta_\star)$, weakly converges to $\mu_\beta$.
\end{proposition}

\begin{proof}
Under the product Lebesgue measure, the two empirical processes associated with the $y$ and $z$ variables are independent and each has asymptotic covariance one half of that of white noise, so
\begin{equation*}
    (\omega_N^\pm, H_N^\pm/N) \xrightarrow[law]{N\to\infty}(\xi,\wick{\brak{\xi,G\ast\xi}}/2)
\end{equation*}
by the same central limit and degenerate U-statistic arguments invoked in the one-species case. Fix $\beta<\beta'<\beta_0$ and choose $p>1$ such that $p\beta'<\beta_0$. Applying \eqref{eq:two-species-partition-bound} at $\beta'$ yields \eqref{eq:two-species-ui} with $\beta'$ in place of $\beta$. Thus the random variables $\exp\pa{\frac\beta N H_N^\pm}$ are uniformly integrable, since they are bounded in $L^{\beta'/\beta}$. Vitali's theorem now gives
\begin{equation*}
    \lim_{N\to \infty}Z_N^\pm(\beta)
    =E\exp\pa{\frac\beta2:\brak{\xi,G\ast\xi}:}
    =\dettwo(I-\beta G)^{-1/2}.
\end{equation*}
The same argument, after multiplication by a bounded continuous functional of $\omega_N^\pm$, proves convergence of the Gibbs measures to $\mu_\beta$.
\end{proof}

Finally, some comments on related literature. Recognizing that $H_N$ is a degenerate U-statistic in principle opens the way to that well-developed theory. However, even the sharpest large deviation results for U-statistics of this kind are not sufficient to establish the uniform bound \eqref{eq:ui}, as it is revealed by a close inspection of the arguments in \cite{ginelatalazinn00,ginelatalazinn01,major}. The issue is clearly the singularity of the interaction $G$, which needs to be regularized for applying those results, shifting the problem to an essentially equivalent uniform control on the removed singular part. As a result on a U-statistic, \cref{thm:main} must be regarded as an exception depending on the particular structure of the interaction kernel.

The point vortex statistical mechanics is closely related to that of 2D Coulomb gas, because the point vortex Hamiltonian (the kinetic energy of vortices) is the potential energy of point charges (whose velocities have Maxwell distribution). 
However, the finite-volume negative temperature ensembles considered in this note are motivated by a specific fluid dynamical interest, and negative temperatures do not appear in 2D Coulomb gas theory, for which I refer to \cite{serfaty15,serfaty17,serfaty18,serfaty21,serfaty23,serfaty26}.

In conclusion, it is natural to wonder whether the perturbative argument of this note can be improved or overcome, reaching the full integrability range $\beta<8\pi$. 
It is worth mentioning that the Gaussian measure $\mu_\beta$ is well-defined for a larger interval $\beta<4\pi^2$ than the integrability one $\beta<8\pi$, and this may already suggest that regularizing the vortex interaction at scale $\epsilon=\epsilon(N)$ and taking a joint limit (as in \cite{bpp}) could be possible even up to $\beta<4\pi^2$. 
In any case, it appears that establishing \eqref{eq:ui} up to $8\pi$ requires a substantial improvement that goes beyond perturbative approaches.

\begin{acknowledgements}
    I wish to thank Eliseo Luongo and Marco Romito for insightful conversations on the topic, and the former for his feedback on a preliminary version of this work.
\end{acknowledgements}

\bibliographystyle{plain}

\begin{thebibliography}{10}

\bibitem{abraham2017}
Eitan Abraham and Oliver Penrose.
\newblock Physics of negative absolute temperatures.
\newblock {\em Physical Review E}, 95(1):012125, 2017.

\bibitem{albeveriocruzeiro}
Sergio Albeverio and Ana-Bela Cruzeiro.
\newblock Global flows with invariant ({Gibbs}) measures for {Euler} and {Navier}-{Stokes} two dimensional fluids.
\newblock {\em Commun. Math. Phys.}, 129(3):431--444, 1990.

\bibitem{albeverioribeiro}
Sergio Albeverio and Raphael H{\o}egh-Krohn.
\newblock Stochastic flows with stationary distribution for two-dimensional inviscid fluids.
\newblock {\em Stochastic Processes Appl.}, 31(1):1--31, 1989.

\bibitem{serfaty21}
Scott Armstrong and Sylvia Serfaty.
\newblock Local laws and rigidity for {Coulomb} gases at any temperature.
\newblock {\em Ann. Probab.}, 49(1):46--121, 2021.

\bibitem{bpp}
G.~Benfatto, P.~Picco, and M.~Pulvirenti.
\newblock On the invariant measures for the two-dimensional {Euler} flow.
\newblock {\em J. Stat. Phys.}, 46(3-4):729--742, 1987.

\bibitem{cagliotiI}
E.~Caglioti, P.~L. Lions, C.~Marchioro, and M.~Pulvirenti.
\newblock A special class of stationary flows for two-dimensional {Euler} equations: {A} statistical mechanics description.
\newblock {\em Commun. Math. Phys.}, 143(3):501--525, 1992.

\bibitem{cagliotiII}
E.~Caglioti, P.~L. Lions, C.~Marchioro, and M.~Pulvirenti.
\newblock A special class of stationary flows for two-dimensional {Euler} equations: {A} statistical mechanics description. {II}.
\newblock {\em Commun. Math. Phys.}, 174(2):229--260, 1995.

\bibitem{dynkin}
E.~B. Dynkin and A.~Mandelbaum.
\newblock Symmetric statistics, {Poisson} point processes, and multiple {Wiener} integrals.
\newblock {\em Ann. Stat.}, 11:739--745, 1983.

\bibitem{eyinkspohn}
G.~L. Eyink and H.~Spohn.
\newblock Negative-temperature states and large-scale, long-lived vortices in two-dimensional turbulence.
\newblock {\em J. Stat. Phys.}, 70(3-4):833--886, 1993.

\bibitem{ginelatalazinn01}
Evarist Gin{\'e}, Stanislaw Kwapie{\'n}, Rafa{\l} Lata{\l}a, and Joel Zinn.
\newblock The {LIL} for canonical {{\(U\)}}-statistics of order 2.
\newblock {\em Ann. Probab.}, 29(1):520--557, 2001.

\bibitem{ginelatalazinn00}
Evarist Gin{\'e}, Rafa{\l} Lata{\l}a, and Joel Zinn.
\newblock Exponential and moment inequalities for {{\(U\)}}-statistics.
\newblock In {\em High dimensional probability II. 2nd international conference, Univ. of Washington, DC, USA, August 1--6, 1999}, pages 13--38. Boston, MA: Birkh{\"a}user, 2000.

\bibitem{grottoluongoromito}
Francesco Grotto, Eliseo Luongo, and Marco Romito.
\newblock Gibbs equilibrium fluctuations of point vortex dynamics.
\newblock {\em Ann. Appl. Probab.}, 34(6):5426--5461, 2024.

\bibitem{grottoromito}
Francesco Grotto and Marco Romito.
\newblock A central limit theorem for {Gibbsian} invariant measures of {2D} {Euler} equations.
\newblock {\em Commun. Math. Phys.}, 376(3):2197--2228, 2020.

\bibitem{grottoromitosd}
Francesco Grotto and Marco Romito.
\newblock Decay of correlation rate in the mean field limit of point vortices ensembles.
\newblock {\em Stoch. Dyn.}, 20(6):16, 2020.
\newblock Id/No 2040009.

\bibitem{gugugu}
Guangze Gu, Changfeng Gui, Yeyao Hu, and Qinfeng Li.
\newblock Uniqueness and symmetry for the mean field equation on arbitrary flat tori.
\newblock {\em Int. Math. Res. Not.}, 2021(24):18812--18827, 2021.

\bibitem{gumora}
Changfeng Gui and Amir Moradifam.
\newblock Symmetry of solutions of a mean field equation on flat tori.
\newblock {\em Int. Math. Res. Not.}, 2019(3):799--809, 2019.

\bibitem{kiessling}
Michael K.-H. Kiessling.
\newblock Statistical mechanics of classical particles with logarithmic interactions.
\newblock {\em Commun. Pure Appl. Math.}, 46(1):27--56, 1993.

\bibitem{kiesslinglebowitz}
Michael K.-H. Kiessling and Joel~L. Lebowitz.
\newblock The micro-canonical point vortex ensemble: {Beyond} equivalence.
\newblock {\em Lett. Math. Phys.}, 42(1):43--58, 1997.

\bibitem{serfaty18}
Thomas Lebl{\'e} and Sylvia Serfaty.
\newblock Fluctuations of two dimensional {Coulomb} gases.
\newblock {\em Geom. Funct. Anal.}, 28(2):443--508, 2018.

\bibitem{serfaty17}
Thomas Lebl{\'e}, Sylvia Serfaty, and Ofer Zeitouni.
\newblock Large deviations for the two-dimensional two-component plasma.
\newblock {\em Commun. Math. Phys.}, 350(1):301--360, 2017.

\bibitem{lionsbook}
Pierre-Louis Lions.
\newblock {\em On {Euler} equations and statistical physics}.
\newblock Pisa: Scuola Normale Superiore, Classe di Scienze, 1998.

\bibitem{major}
P{\'e}ter Major.
\newblock On a multivariate version of {Bernstein}'s inequality.
\newblock {\em Electron. J. Probab.}, 12:966--988, 2007.

\bibitem{onsager1949}
Lars Onsager.
\newblock Statistical hydrodynamics.
\newblock {\em Il Nuovo Cimento (1943-1954)}, 6(Suppl 2):279--287, 1949.

\bibitem{purcell1951}
Edward~M Purcell and Robert~V Pound.
\newblock A nuclear spin system at negative temperature.
\newblock {\em Physical Review}, 81(2):279, 1951.

\bibitem{serfaty26}
Matthew Rosenzweig and Sylvia Serfaty.
\newblock Sharp commutator estimates of all order for {Coulomb} and {Riesz} modulated energies.
\newblock {\em Commun. Pure Appl. Math.}, 79(2):207--292, 2026.

\bibitem{serfaty15}
Etienne Sandier and Sylvia Serfaty.
\newblock {2D} {Coulomb} gases and the renormalized energy.
\newblock {\em Ann. Probab.}, 43(4):2026--2083, 2015.

\bibitem{serfaty23}
Sylvia Serfaty.
\newblock Gaussian fluctuations and free energy expansion for {Coulomb} gases at any temperature.
\newblock {\em Ann. Inst. Henri Poincar{\'e}, Probab. Stat.}, 59(2):1074--1142, 2023.

\end{thebibliography}

\end{document}